\documentclass[reqno]{amsart}

\usepackage{amsmath}
\usepackage{amsfonts}
\usepackage{amssymb}
\usepackage{amsthm}
\usepackage{mathenv}
\usepackage{enumerate}
\usepackage{graphicx}
\usepackage[all]{xy}

\theoremstyle{plain} 
\newtheorem{theorem}{Theorem}[section]
\newtheorem{prop}[theorem]{Proposition}
\newtheorem{lem}[theorem]{Lemma}

\theoremstyle{definition} 
\newtheorem{defn}[theorem]{Definition}
\newtheorem{rem}[theorem]{Remark}
\newtheorem{nota}[theorem]{Notation}
\newtheorem{example}[theorem]{Example}

\numberwithin{equation}{section}

\newcommand{\B}[1]{\mathbb{#1}}
\newcommand{\rst}[1]{\ensuremath{{\mathbin\upharpoonright}%
\raise-.5ex\hbox{$#1$}}} 
\newcommand{\norm}[2]{\left\| {#1}\right\| _{#2}}

\newcommand{\cl}[1]{{\rm cl} ( #1 )} % closure of a set
\newcommand{\cc}{\subset\!\!\!\subset}

\newcommand{\tr}{{\mathcal L}}

\newcommand{\RS}{{\hat{\mathbb{C}}}}
\newcommand{\HR}{H^2(D^\infty_R)}%{\hat{\mathbb{C}}}}
\newcommand{\HRnull}{H^2_0(D^\infty_R)}
\newcommand{\monster}{H^2(D_r)\oplus H^2_0(D^\infty_R)}

\newcommand{\PRplus}{\Pi^{(R)}_+}

\newcommand{\PRminus}{\Pi^{(R)}_-}

\title[Spectral structure of transfer operators]
{Spectral structure of transfer operators for expanding circle maps}
\author{Oscar F.~ Bandtlow, Wolfram Just, and Julia Slipantschuk}
\address{School of Mathematical Sciences, Queen Mary University of London, Mile End Road,
London E1 4NS, UK}
\email{o.bandtlow@qmul.ac.uk, \,w.just@qmul.ac.uk, \,j.slipantschuk@qmul.ac.uk,}
\begin{document}

\begin{abstract}
We explicitly determine the spectrum of transfer operators (acting on spaces
of holomorphic functions) associated to analytic
expanding circle maps arising from finite Blaschke products. 
This is achieved by deriving a convenient natural representation of the 
respective adjoint operators.
\end{abstract}

\maketitle

\section{Introduction}
One of the major strands of 
modern ergodic theory is to exploit the rich 
links between dynamical systems theory and 
functional analysis, making the powerful tools of the latter available for the 
benefit of understanding complex dynamical behaviour. 
In classical ergodic theory, composition operators 
occur naturally as basic objects for formulating concepts such as ergodicity or mixing
\cite{Walters2000}.
%Basic objects in classical ergodic theory 
%for formulating concepts such as ergodicity or mixing are 
%composition operators, known in this context as Koopman operators, 
%see \cite{Walters2000}.
These operators, 
known in this context as Koopman operators, are the formal adjoints of transfer operators, 
the spectral data of which 
provide insight into fine statistical properties 
of the underlying dynamical systems, such as rates of mixing (see, for 
example, \cite{BoyaGora:97}).

In the literature, the term `composition operator' is mostly used to refer to
compositions with 
%operators induced by 
analytic functions mapping a disk into itself, 
a setting in which  operator-theoretic properties such as boundedness, 
compactness, and most importantly explicit spectral information are well-established
(good references are \cite{Shapiro1993} or the encyclopedia on the subject 
\cite{CowM1995}).
The purpose of the present
article is to demonstrate that in a particular analytic setting, 
the spectra of transfer operators can be deduced directly from certain 
composition operators. 

Let $\tau$ be a real analytic expanding map 
on the circle and $\{\phi_k\colon k=1,\ldots, K \}$ the set of local inverse 
branches of $\tau$. 
Then the associated transfer operator $\mathcal{L}$ (also referred to
as Ruelle-Perron-Frobenius or simply Ruelle operator),
defined by 
\begin{equation}\label{eq:L_intro} 
 (\mathcal{L} f)(z) = \sum_{k=1}^{K} \phi'_k(z)(f\circ\phi_k)(z),
\end{equation} 
preserves and acts compactly on certain spaces of holomorphic functions 
\cite{ruelle76,RUGH2002,Nr.2}.

The spectrum of $\tr$ can be understood by passing to its (Banach space) adjoint operator $\tr^*$.
This strategy has been explored in the context 
of Ruelle operators acting on the space of functions locally analytic 
on the Julia set of a rational function $R$, see 
\cite{Baladia, Levin1991, Levin1994a, Ushiki2000}; in particular, explicit expressions
for Fredholm determinants of certain Ruelle operators have been derived.
In our setting of analytic expanding circle maps, we adopt a similar approach, 
that is, we analyse the spectrum of $\tr$ by deriving a natural explicit 
representation of $\tr^*$ (Proposition \ref{prop:adj}). %Thereby, we
%extend our previous work \cite{Nr.2}, where we directly determined
%the spectrum of $\tr$ for a certain family of analytic circle maps.
As a by-product we obtain a more conceptual proof of 
results in \cite{Nr.2}, where 
the spectrum of $\tr$ for a certain family of analytic circle maps was
determined using a block-triangular matrix representation. 

In the spirit of approaching problems in the world of real 
numbers by making recourse to complex numbers, we 
consider finite Blaschke products, a class of rational maps on 
the Riemann sphere $\hat{\mathbb{C}}=\mathbb{C}\cup\{\infty\}$ preserving
the unit circle. One of the striking features of 
Blaschke products is that they partition
the Riemann sphere into simple 
dynamically invariant regions: the unit circle, the unit disk
and the exterior disk in $\RS$. As a consequence, 
the spectrum of $\tr^*$ can be determined by studying the spectrum 
of composition operators on spaces of holomorphic
functions on these dynamically invariant regions.

Our main result (Theorem \ref{thm:main}) can be summarized as follows.
Let $B$ be a finite Blaschke product such that its restriction $\tau$ to the
unit circle $\mathbb{T} = \{z:|z|=1\}$ is expanding. Denote by
$H^2(A)$ the Hardy-Hilbert space of functions which are holomorphic on some suitable
annulus $A$ (containing $\mathbb{T}$) and square integrable on its boundary 
$\partial A$ (see Definition \ref{defn:HardyHilbert}). 
Then the transfer operator $\tr$ associated to $\tau$ is compact on $H^2(A)$, with 
spectrum
$$\sigma(\tr) = 
\left \{ \lambda(z_0) ^n : n\in \mathbb{N}_0\right \} \cup 
 \{ \overline{\lambda(z_0)} ^n : n\in \mathbb{N} \}
\cup \left \{0 \right \},$$
where $\lambda(z_0)$ is the multiplier of the unique attracting 
fixed point $z_0$ of $B$  in the unit disk.
%$z_0$ and $1/\overline{z_0}$ are the unique attracting fixed
%points of $B$ in each of the components of $\RS\setminus\overline{A}$.
This implies that for finite Blaschke products which give rise to analytic expanding circle maps, 
the derivative of the fixed point in the unit disk completely determines 
the spectrum of $\mathcal{L}$.

The paper is organized as follows. 
In Section~\ref{sec:Prelim}, we review basic definitions and 
facts about Hardy-Hilbert spaces on annuli. The following 
%define %the objects of interest 
Section~\ref{sec:CircTrans} 
is devoted to analytic expanding circle maps and their corresponding 
transfer operators. %on appropriate Hardy Hilbert spaces, 
%and some their properties. 
In Section \ref{sec:Adjoint}, we explicitly derive the structure of the 
corresponding adjoint operators after having established a suitable 
representation of the dual space. This structure is then used in 
Section \ref{sec:Blaschke} in order
to obtain the spectrum of transfer operators associated to analytic expanding
circle maps arising from finite Blaschke products, %and 
thus proving our main result.

%%%%%%%%%% Hardy-Hilbert spaces %%%%%%%%%%
\section{Hardy-Hilbert spaces}\label{sec:Prelim}
Throughout this article $\hat{\mathbb{C}}=\mathbb{C}\cup\{\infty\}$ 
denotes the one point compactification of $\mathbb{C}$. For $U\subset
\RS$, we write $\cl{U}$ to denote the closure of $U$ in $\RS$. 
For $r>0$ we use 
\begin{align*}
\mathbb{T}_r&=\{z\in\mathbb{C}:|z|=r\}\,,\\
\mathbb{T}_{\hphantom{r}}&=\mathbb{T}_1
\end{align*}
to denote circles centred at $0$, and 
\begin{align*} 
D_r&=\{z\in \mathbb{C}:|z|<r\}\,,\\
D^\infty_r&=\{z\in \hat{\mathbb{C}}: |z|>r\}\,,\\
\mathbb{D}&=D_1
\end{align*}
to denote disks centred at $0$ and $\infty$. 
Given $r<1<R$ %let $A_{r,R}$ 
the symbol 
\[ A_{r,R} = \{z \in \mathbb{C}:r<|z|<R\} \]
will denote an open annulus containing the unit circle $\mathbb T$. 

We write $L^p( \mathbb{T}_r) = L^p( \mathbb{T}_r,d\theta/2\pi)$ 
with $1\leq p < \infty$ for the Banach space of \mbox{$p$-integrable} functions 
with respect to normalized one-dimensional Lebesgue measure
on $\mathbb{T}_r$.
Finally, for $U$ an open subset of $\RS$ we use 
$\operatorname{Hol}(U)$ for the space of holomorphic functions on $U$.

Hardy-Hilbert spaces on disks and annuli will provide a convenient
setting for our analysis. We briefly recall their properties in the following. 

%%%%%%% Function spaces %%%%%%%%%%%%%%%%%
\begin{defn}[Hardy-Hilbert spaces]\label{defn:HardyHilbert}
%Let $D_r$ and $A_{r,R}$ be as above and 
For $\rho>0$ and $f: \mathbb{T}_\rho \to \mathbb{C}$ 
write
\[M_\rho(f)=\int_{0}^{2\pi}|f(\rho e^{i\theta})|^2\, 
\frac{d\theta}{2\pi}.\]
Then the Hardy-Hilbert spaces on $D_r$ and $A_{r,R}$ are given by
\begin{equation*}
H^2(D_r) = \left \{f\in \operatorname{Hol}(D_r):\;
\sup_{\rho\nearrow r} M_\rho(f)<\infty \right\}, 
\end{equation*}
and 
\begin{eqnarray*}
 H^2(A_{r,R}) &=& \left\{f\in \operatorname{Hol}(A_{r,R}):\;
\sup_{\rho\nearrow R} M_\rho(f)+ 
\sup_{\rho\searrow r} M_\rho(f) <\infty \right\}.
\end{eqnarray*}
The Hardy-Hilbert space on the exterior disk $D^\infty_R$ is defined accordingly, 
that is $f\in H^2(D^\infty_R)$ if $f\in \operatorname{Hol}(D^\infty_R)$ (or,
equivalently, $f\circ\varsigma$ holomorphic on $D_{1/R}$ with $\varsigma(z)=1/z$) 
and $\sup_{\rho\searrow R} M_\rho(f) <\infty$. Finally, 
$\HRnull \subset \HR$ 
denotes the subspace of
functions vanishing at infinity.
\end{defn}

A comprehensive account of Hardy spaces over general domains is given
in the classic text \cite{Duren}. A crisp treatment of Hardy spaces on the
unit disk can be found in \cite[Chapter 17]{RudinRCA}), while a good
reference for Hardy spaces on annuli is \cite{Sarason1965}. We shall
now collect a number of results which will be useful in what follows. 
 
Any function in $H^2(U)$, where $U$ is a disk or an annulus, can be extended 
to the boundary in the following sense.
For $f\in H^2(D_r)$ there is an $f^*\in L^2(\mathbb{T}_r)$ 
such that
\begin{equation*}%\label{eq:boundaryFunc}
 \lim_{\rho\nearrow r} f(\rho e^{i\theta}) = f^*(re^{i\theta}) \quad \text{for a.e.~$\theta$},
\end{equation*}
and analogously for $f\in \HR$. Similarly, for $f\in H^2(A_{r,R})$ there are
$f^*_1\in L^2(\mathbb{T}_r)$ and $f^*_2\in L^2(\mathbb{T}_R)$,
with $\lim_{\rho\searrow r} f(\rho e^{i\theta}) = f_1^*(re^{i\theta})$
and $\lim_{\rho\nearrow R} f(\rho e^{i\theta}) = f_2^*(Re^{i\theta})$
for a.e.~$\theta$.  
It turns out that the spaces $H^2(U)$ 
are Hilbert spaces with inner products 
\[(f,g)_{H^2(D_r)} = 
          \int_{0}^{2\pi} f^*(re^{i\theta})\overline{g^*(re^{i\theta})}\;\frac{d\theta}{2\pi}\]
and
\[(f,g)_{H^2(A_{r,R})} = \int_{0}^{2\pi} f_1^*(Re^{i\theta})\overline{g_1^*(Re^{i\theta})}\;\frac{d\theta}{2\pi}+
         \int_{0}^{2\pi}
         f^*_2(re^{i\theta})\overline{g^*_2(re^{i\theta})}\;\frac{d\theta}{2\pi}. 
\]
Similarly, for $\HR$.
% \item (Norm convergence) For $f\in H^2(D_{r})$ and $\rho \in (0,1)$ define
% $f_\rho(z) = f(\rho z)$, then $f_\rho \in L^2(\mathbb{T}_r)$ and
% $\norm{f_\rho-f^*}{L^2(\mathbb{T}_r)}\to 0$ as $\rho\nearrow 1$. Similar
% arguments apply to $\HA$ and $\HR$.
% \end{enumerate}

\begin{nota}
%We shall adopt the convenient and common notation from the literature and 
%write $f(z)$ instead of $f^*(z)$ for $z$ on the boundary of the domain.
In order to avoid cumbersome notation, we shall write
$f(z)$ instead of $f^*(z)$ for $z$ on the boundary of the domain. 
%meaning that on the boundary $f$ is defined
%almost everywhere as $f^*$.
\end{nota}

\begin{rem}\label{rem:Basis}
It is not difficult to see that $\mathcal{E} = \{e_n: n\in
\mathbb{Z}\}$ where 
\begin{equation}\label{eq:en}
 e_n(z) = \frac{z^n}{d_n}, \quad \text{with } d_n = \sqrt{r^{2n}+R^{2n}}
\end{equation}
is an orthonormal basis for $H^2(A_{r,R})$. In particular, it follows
that 
$f\in \operatorname{Hol}(A_{r,R})$ is in $H^2(A_{r,R})$ if and only if 
$f(z)=\sum_{n=-\infty}^{\infty} c_n e_n(z)$ with
$\sum_{n=-\infty}^{\infty} |c_n|^2< \infty$, where
the coefficients are given by 
$c_n=c_n(f) = (f,e_n)_{H^2(A_{r,R})}$. Note also that  
$\norm{f}{H^2(A_{r,R})}^2 =\sum_{n=-\infty}^{\infty} |c_n|^2$.
\end{rem}

% We also state a simple lemma, which will be useful in the sequel.
% \begin{lem}\label{lem:homotopy}
% Let $f$ be holomorphic in $A_{r,R}$, assume that its boundary function $f^*$
% defined in \eqref{eq:boundaryFunc} is in $L^1(\mathbb{T}_r)$, 
% and let $\gamma$ be a smooth curve in $A_{r,R}$ homotopic to $\mathbb{T}_r$.
% Then 
% $$\int_{\mathbb{T}_r} f^*(z) \;dz = \int_{\gamma} f(z) \;dz.$$
% \end{lem}
% \begin{proof}
% Let $f_\rho(z) = f(\rho z)$, then
%  $\norm{f_\rho-f^*}{L^1(\mathbb{T}_r)}\to 0$ as $s\searrow 1$, see, for example, \cite[Thm. 17.11]{RudinRCA}.
% The conclusion follows from the fact that
% $$\int_{\mathbb{T}_{\rho r}} f(z) \;dz = \int_{\gamma} f(z) \;dz,$$ by
% homotopy.
% \end{proof}

%%%%%%%%%% Circle maps and transfer operators %%%%%%%%%%
\section{Circle maps and transfer operators}\label{sec:CircTrans}
The purpose of this section is to establish compactness of transfer operators on
Hardy-Hilbert spaces for analytic expanding circle maps, defined as follows.
%These results are standard but we include them for the convenience of
%the reader.
\begin{defn}
 We say $\tau \colon \mathbb{T} \to \mathbb{T} $ is an \textit{analytic expanding circle map}
 if 
 \begin{enumerate}[(i)]
  \item $\tau$ is analytic on $\mathbb{T}$; 
  \item $|\tau'(z)|>1$ for all $z\in \mathbb{T}$.    
 \end{enumerate}        
\end{defn}

In particular, $\tau$ is a $K$-fold covering for some $K>1$.
With a slight abuse of notation we continue to write $\tau$ for
its holomorphic extension to an open annulus $A_{r,R}$ for $r<1<R$ 
%containing $\mathbb{T}$ 
and let
\[ \mathcal A=\left \{ A_{r,R}: \tau \text{ and } 
1/\tau \text{ holomorphic on }A_{r,R}  \right \}. \]

Before proceeding to the definition of transfer operator associated to $\tau$ 
we require some more notation. Given two open subsets $U$ and $V$ of $\mathbb C$ we write 
\[ U\cc V\]  
if $\cl{U}$ is a compact subset of $V$. 

By the expansivity of $\tau$ and 
\cite[Lemma 2.2]{Nr.2}, we can choose $A_0, A'$ and $A$ in $\mathcal{A} $ 
with 
\begin{equation}\label{eq:AA'}
 A_0\cc A' \cc A \text{ and } \tau(\partial A_0)\cap \cl{A} = \emptyset.
\end{equation}

Given an analytic expanding circle map 
$\tau$, we associate with it a transfer operator $\tr$ by setting 
\begin{equation}\label{eq:ldef}
(\tr f)(z) = \sum_{k=1}^K \phi'_k(z)(f\circ\phi_k)(z) \,,
\end{equation}
where $\phi_k$ denotes the $k$-th local inverse of $\tau$.
% In particular, this implies that $\phi_k(A) \cc A'$ for all $k$.

As we shall see presently, the above choices of the annuli %$A,A'\in\mathcal{A}$ 
guarantee that
$\tr$ is a well-defined linear operator which maps $H^2(A)$ compactly to itself.
In order to show this, we shall employ a 
factorization argument, similar to the ones used  
in \cite{BJ08,Nr.2}. 
Let $H^\infty(A')$ be the Banach space of bounded
holomorphic functions on %some annulus 
$A'$ equipped with the supremum norm. We can write $\tr = \tilde{\tr} \mathcal{J}$, 
where $\tilde{\tr}\colon H^\infty(A')\to H^2(A)$ is a
lifted transfer operator given by the same functional expression \eqref{eq:ldef}
and $\mathcal{J}\colon H^2(A)\to H^\infty(A')$ is the canonical
embedding 
\[\xymatrix{
& H^\infty(A') \ar[rd]^{\tilde{\tr}} &\\
&H^{2}(A) \ar@{^{(}->}[u]^{\mathcal{J}} \ar[r]_{\tr}  &H^2(A) }\]
We use $H^\infty(A')$ instead of $H^2(A')$ as this choice allows
for an easy proof of continuity of $\tilde{\tr}$ in Lemma \ref{lem:ltilde}.

Let $R,R'$ denote the radii of the circles forming the `exterior' boundaries, 
and $r,r'$ the radii of the circles forming the `interior' boundaries
of $A$ and $A'$, respectively, that is, $A=A_{r,R}$ and $A'=A_{r',R'}$.
%We can now state continuity for $\tilde \tr$.
\begin{lem}\label{lem:ltilde}
The transfer operator $\tilde{\tr}$ given by $\eqref{eq:ldef}$ maps $H^\infty(A')$ 
continuously to $H^2(A)$.
\end{lem}

\begin{proof}
We can factorize 
$\tilde{\mathcal{L}}$ as $\tilde{\mathcal{L}} = \hat{\mathcal{J}}\hat{\mathcal{\mathcal{L}}}$,
where $\hat{\mathcal{L}}\colon H^\infty(A')\to H^\infty(A)$, given by the functional
expression \eqref{eq:ldef}, is continuous
by \cite[Lemma~2.5]{Nr.2}, and $\hat{\mathcal{J}}\colon H^\infty(A)
\hookrightarrow H^2(A)$ is the canonical embedding. 
\end{proof}

Next, we establish compactness of $\mathcal{J}\colon H^2(A)\hookrightarrow H^\infty(A')$ given by 
\begin{equation*}%\label{eq:J}
(\mathcal{J}f)(z) = f(z) \quad \text{ for } z\in A'.
\end{equation*}
Let $\{e_n : n \in \mathbb{Z} \}$ be the orthonormal basis for $H^2(A)$ given by \eqref{eq:en},
then any $f\in H^2(A)$ can be uniquely expressed as $f = \sum_{n\in \mathbb{Z}} c_n(f) e_n$. 
For $N\in \mathbb{N}$ define the finite rank operator $\mathcal{J}_N\colon H^2(A)\to H^\infty(A')$ by 
\begin{equation*}%\label{eq:JN}
(\mathcal{J}_Nf)(z) = \sum_{n=-N+1}^{N-1} c_n(f) e_n(z) \quad \text{ for } z\in A'.
\end{equation*}
\begin{lem}\label{lem:Jemb}
Let $\mathcal{J}$ and $\mathcal{J}_N$ be as above. Then
\begin{equation*}
\lim_{N\to \infty} \norm{\mathcal{J}-\mathcal{J}_N}{H^2(A)\to H^\infty(A')} = 0.
\end{equation*}
In particular, the embedding $\mathcal{J}$ is compact.
\end{lem}
\begin{proof}
For $z\in A'$, it follows by the Cauchy-Schwarz inequality that
\begin{align*}
 |(\mathcal{J}f)(z)-(\mathcal{J}_Nf)(z)| 
&\leq \left(\sum_{|n|\geq N} |c_n(f)|^2\right )^{1/2}
\left(\sum_{|n|\geq N} |e_n(z)|^2\right )^{1/2}\\
 &\leq\norm{f}{H^2(A)}
\left( \sum_{|n| \geq N} \frac{|z^n|^2}{r^{2n}+R^{2n}}\right)^{1/2}\\
&\leq \norm{f}{H^2(A)} \left( \sum_{n\geq N} \left|\frac{z}{R}\right|^{2n} + 
\sum_{n\geq N} \left|\frac{r}{z}\right|^{2n} \right )^{1/2}.
\end{align*}
Thus, 
\begin{align*}
 \norm{\mathcal{J}f-\mathcal{J}_Nf}{H^\infty(A')} 
\leq \norm{f}{H^2(A)}  \left(\left(\frac{R'}{R} \right)^{2N}\frac{1}{1-(\frac{R'}{R})^2}
+ \left(\frac{r}{r'} \right)^{2N}\frac{1}{1-(\frac{r}{r'})^2} \right)^{1/2},
\end{align*}
and the assertion follows.
\end{proof}

The factorization $\tr = \tilde{\tr}\mathcal{J}$
together with Lemmas \ref{lem:ltilde} and \ref{lem:Jemb} 
now imply the following result.
\begin{prop}\label{prop:Lcompact}
 The transfer operator $\tr\colon H^2(A)\to H^2(A)$ in \eqref{eq:ldef} is compact.
\end{prop}
\begin{rem}
  Closer inspection of Lemma~\ref{lem:Jemb} reveals that
  the singular values of $\mathcal{J}$ decay at an exponential
  rate. Thus $\mathcal J$ and hence $\tr$ are trace-class. In fact,
  using results from \cite{expoclass} it is possible to show that both
  the singular values and the eigenvalues of $\tr$ 
decay at an exponential rate, a property that $\tr$ shares with
  other transfer operators arising from analytic maps (see, for
  example, \cite{fried,bjadvances}). 
\end{rem}

%%%%%%%% Adjoint operator: NEW PROOF (J and J^-1 interchanged) %%%%%%%%%%%%%%%%%%%%%%%%%
\section{Adjoint operator} \label{sec:Adjoint}
A central step in showing our main result is to find an appropriate representation
of the dual space on which the adjoint of the transfer operator has a simple structure. 

For the remainder of this section we set $A = A_{r,R}$
and denote by $H^2(A)^*$ the strong dual of $H^2(A)$, that is, the
space of continuous linear functionals on $H^2(A)$ equipped with the
topology of uniform convergence on the unit ball. 
We will show that $H^2(A)^*$ is isomorphic to the topological direct sum
$H^2(D_r)\oplus \HRnull$, equipped with the norm 
%$\norm{(h_1,h_2)}{H^2(D_r)\oplus \HRnull}=\sqrt{\norm{h_1}{H^2(D_r)}^2+\norm{h_2}{\HRnull}^2}$.
$\norm{(h_1,h_2)}{}^2 = \norm{h_1}{H^2(D_r)}^2+\norm{h_2}{\HRnull}^2$.
Similar representations of the duals of Hardy spaces 
for multiply connected regions can be found in 
\cite[Proposition 3]{Royden1988}. The present set-up is sufficiently
simple to allow for a short proof of the representation. 

%In the following lemma we shall use the map $K_z\colon A\to \mathbb{C}$ be defined by $K_z(w) = 1/(z-w)$ for 
%$z\in \hat{\mathbb{C}}\setminus \overline{A}$.
 
\begin{prop}\label{lem:iso_new}
The dual space $H^2(A)^*$ is isomorphic to $H^2(D_r)\oplus \HRnull$ with the isomorphism
given by
\begin{eqnarray*}  
  J\colon H^2(D_r)\oplus \HRnull &\rightarrow& H^2(A)^*\\ 
      (h_1, h_2) &\mapsto & l,
\end{eqnarray*}
where
\begin{equation}\label{eq:lKz}
 l(f) = \frac{1}{2\pi i} \int_{\B{T}_r} f(z)h_1(z) \;dz + 
\frac{1}{2\pi i} \int_{\B{T}_R} f(z)h_2(z) \;dz \quad (f\in H^2(A)).
%h_1(z) = l(-K_z) \text{ for } z\in D_r \text{ and } h_2(z) = l(K_z) \text{ for } 
%z\in D^\infty_R. 
\end{equation}
\end{prop}

\begin{proof} 
We will first show that \eqref{eq:lKz} defines a continuous functional
$l\in H^2(A)^*$ and
that $J$ is a bounded linear operator.
In order to see this note that 
for any $(h_1,h_2)\in \monster$ the linear functional $l=J(h_1,h_2)$
is bounded, since for any $f\in H^2(A)$ with $\norm{f}{H^2(A)}\leq 1$ 
\[|l(f)| \leq \left(r\norm{h_1}{H^2(D_r)} + 
R\norm{h_2}{\HRnull}\right).\]
It follows that  
\[\norm{J(h_1,h_2)}{H^2(A)^*}\leq 
\sqrt{r^2+R^2}\sqrt{\norm{h_1}{H^2(D_r)}^2+\norm{h_2}{\HRnull}^2}\] and
$\norm{J}{H^2(D_r)\oplus \HRnull \to H^2(A)^*} \leq \sqrt{r^2+R^2}$.
Hence, $J$ is well defined and bounded.

%\medskip

For injectivity, we suppose that 
$l=J(h_1,h_2) = 0$ and show that $h_1 = 0$ and $h_2 = 0$.
In order to see this note that any $(h_1,h_2)\in \monster$ 
can be written $h_1(z) = \sum_{n=0}^{\infty} a_n z^n$
and $h_2(z) = \sum_{n=1}^{\infty} a_{-n} z^{-n}$ with suitable 
%corresponding Laurent 
coefficients $a_n\in \mathbb{C}$. 
Now let 
\[ \mathcal{E} = \{e_n : n \in \mathbb{Z} \}\quad \text{with}\quad e_n(z) = \frac{z^n}{d_n}\] 
denote the orthonormal basis of $H^2(A)$ given in
Remark~\ref{rem:Basis}. 
A short calculation using Lebesgue dominated convergence shows 
that 
%, with $e_m\in \mathcal{E}$ given by \eqref{eq:en}, %with $m\in \mathbb{Z}$ 
\begin{equation}\label{eq:an_dn}
0 = (J(h_1,h_2))(e_{n}) = \frac{a_{-n-1}}{d_{n}} \quad \text{for all } 
n\in \mathbb{Z},
\end{equation}
which implies $h_1 = 0$ and $h_2 = 0$. Thus $J$ is injective.

% \medskip

Finally, in order to show that $J$ is surjective, fix $l \in H^2(A)^*$. 
We will construct $(h_1,h_2)\in \monster$ such that $J(h_1,h_2)=l$.

By the Riesz representation 
theorem 
there is a 
unique $g \in H^2(A)$ such that
$l(f) = (f,g)_{H^2(A)}$ for all $f\in H^2(A)$. 
Moreover, $g$ can be uniquely expressed as 
$g = \sum_{n\in \mathbb{Z}} c_n(g) e_n$.
Now define 
\begin{equation}\label{eq:h1h2}
\begin{split}
h_1(z) &= \sum_{n=0}^{\infty}{\overline{c_{-n-1}(g)}} d_{-n-1} z^{n}  
\quad \text{for } z\in D_r,\\  
h_2(z) &= \sum_{n=1}^{\infty} \overline{c_{n-1}(g)}d_{n-1}z^{-n}
\quad \; \; \text{for } z\in D^\infty_R.
\end{split} 
\end{equation}
Using $\norm{g}{H^2(A)}^2 = \sum_{n\in \mathbb{Z}} |c_n(g)|^2 < \infty$, it follows 
that $h_1 \in H^2(D_r)$ and $h_2 \in \HRnull$. 
Combining \eqref{eq:an_dn} and \eqref{eq:h1h2} we obtain 
\[(J(h_1,h_2))(e_{n}) = \frac{a_{-n-1}}{d_{n}} = \frac{\overline{c_n(g)}d_{n}}{d_{n}}
= \overline{c_n(g)} = (e_{n},g)_{H^2(A)} \]
for every $n\in \mathbb{Z}$.  
Since the above equality also holds for all finite linear combinations
of elements in $\mathcal{E}$ the continuity of $J$ implies 
\[(J(h_1,h_2))(f) = (f,g)_{H^2(A)} = l(f)\] for all $f\in H^2(A)$. 
Thus $J$ is surjective.
\qedhere
\end{proof}

\begin{rem}
The inverse $J^{-1}$ of $J$ can be obtained using the kernel 
$K_z \in H^2(A)$ defined by $K_z(w) = 1/(z-w)$ for 
$z\in \hat{\mathbb{C}}\setminus \cl{A}$. 
More precisely, $J^{-1}$ is given by $l\mapsto (h_1,h_2)$, where $h_1(z) = l(-K_z)$
for $z\in D_r$ and $h_2(z) = l(K_z)$ for $z\in D^\infty_R$.
\end{rem}

Returning to the setting of Section \ref{sec:CircTrans}, let $\tau$ be an analytic 
expanding circle map and $A=A_{r,R} \in \mathcal{A}$ an annulus satisfying \eqref{eq:AA'} 
such that the associated  transfer operator $\tr\colon H^2(A)\to H^2(A)$ is well defined 
and compact. 
Using the representation of the dual space $H^2(A)^*$ obtained in the
previous lemma, we shall shortly derive 
an explicit form for the adjoint operator of $\tr$. 

Before doing so we require some more notation. 
Define $C^{(r)}\colon H^2(D_r)\to L^2(\mathbb{T}_{r})$ by
\begin{equation}\label{eq:Crho}
(C^{(r)}h)(z)=h(\tau(z)) \quad \text{for } z\in \mathbb{T}_r\,, 
\end{equation}
and $C^{(R)}\colon H^2_0(D^\infty_R)\to L^2(\mathbb{T}_{R})$ by 
\begin{equation}\label{eq:Crho2}
(C^{(R)}h)(z)=h(\tau(z)) \quad \text{for } z\in \mathbb{T}_R\,. 
\end{equation}

It turns out that $C^{(r)}$ and $C^{(R)}$ are compact, the proof of
which relies on the following fact. 

\begin{lem}\label{lem:K}
 Let $K$ be a compact subset of a disk $D$ in $\mathbb{C}$. 
Then there exists a constant $c_K$
depending on $K$ only such that for any $f\in H^2(D)$
\[\sup_{z\in K} |f(z)|\leq c_K \norm{f}{H^2(D)}\,.\]
\end{lem}
\begin{proof}
This follows, for example, from \cite[Lem.~2.9]{BJ08}, or by a
calculation using the Cauchy-Schwarz inequality similar to the proof
of Lemma~\ref{lem:Jemb}. 
\end{proof} 
We now have the following. 
\begin{lem}\label{lem:Cr_compact}
The operators $C^{(r)}$ and $C^{(R)}$ are compact.
\end{lem}
\begin{proof}%[Proof of Lemma \ref{lem:Cr_compact}]
The choice of $A=A_{r,R}$ in \eqref{eq:AA'} implies that 
$r_0 = \sup_{z\in \mathbb{T}_r}|\tau(z)| < r$, and we
can choose a disk $D_{r'}$ with $D_{r_0} \cc D_{r'} \cc D_r$.

Let $\tilde{C}^{(r)}:H^2(D_{r'})\to L^2(\mathbb{T}_r)$ be defined by
the functional expression as in \eqref{eq:Crho}, but now considered on 
$H^2(D_{r'})$. The operator is continuous since
\begin{align*}
 \norm{\tilde{C}^{(r)}h}{L^2(\mathbb{T}_{r})} 
%&= \norm{h\circ \tau}{L^2(\mathbb{T}_{r})}
%&\leq \norm{h\circ \tau}{L^\infty(\mathbb{T}_{r})}
\leq \sup_{z\in \tau(\mathbb{T}_{r})}|h(z)|
&\leq \sup_{z\in \cl{D_{r_0}}}|h(z)|
\leq c_K \norm{h}{H^2(D_{r'})}, 
\end{align*}
where we have used Lemma \ref{lem:K} with $K=\cl{D_{r_0}}.$
The lemma follows since we can write
$C^{(r)} = \tilde{C}^{(r)} \tilde{\mathcal{J}}$ with 
$\tilde{\mathcal{J}}\colon H^2(D_r)\hookrightarrow H^2(D_{r'})$
denoting the canonical embedding, 
%defined by $\tilde{\mathcal{J}}h=h|_{D_{r'}}$, 
which is compact (see, for example,
\cite[Lemma 2.9]{BJ08}).
The argument for $C^{(R)}$ is similar.
\end{proof}

Next, we need to define certain projection operators on  $L^2(\mathbb{T_\rho})$. 
For any $g\in L^2(\mathbb{T_\rho})$ we can write $g(z) = \sum_{n\in \mathbb{Z}} g_n z^n$,
so that $g = g_+ + g_-$ with 
$g_+(z)=\sum_{n=0}^{\infty} g_n z^n$ and  
$g_-(z)=\sum_{n=1}^{\infty} g_{-n} z^{-n}$. % = g_{+}(z) + g_{-}(z)$.
%, where $c_n(g)$ are the Fourier coefficients of $g$. 
Since $\norm{g}{L^2(\mathbb{T_\rho})}^2 = 
\sum_{n=-\infty}^{\infty} |a_n|^2 \rho^{2n}<\infty,$ the functions $g_{+}$ and $g_{-}$ can
be viewed as functions in $H^2(D_\rho)$ and 
$H^2_0(D^\infty_\rho)$, respectively. Then we define the bounded projection operators 
$\Pi^{(\rho)}_+\colon L^2(\mathbb{T}_\rho)\to H^2(D_\rho)$ and
%$\Pi^{(\rho)}_+(g) = g_{+}$ and 
$\Pi^{(\rho)}_-\colon L^2(\mathbb{T}_\rho)\to H^2_0(D^\infty_\rho)$ by
%$\Pi^{(\rho)}_-(g) = g_{-}$.
\begin{equation}\label{eq:P+P-}
\Pi^{(\rho)}_+(g) = g_{+}\quad \text{and} \quad  \Pi^{(\rho)}_-(g) = g_{-}. 
\end{equation}

Finally, let $\tr^*\colon H^2(A)^*\to H^2(A)^*$ denote the adjoint operator of $\tr$ in 
the Banach space sense, that
is, $(\tr^*l)(f) = l(\tr f)$ for all $l\in H^2(A)^*$ and  $f\in H^2(A)$.
The following proposition provides an explicit representation $\tr'$ of $\tr^*$ via 
\[\tr' = J^{-1}\tr^*J,\]
as an operator on $\monster$ 
given by compositions of $C^{(\rho)}$, $\Pi^{(\rho)}_- $
and $\Pi^{(\rho)}_+ $ for $\rho = r, R$.
 
\begin{prop}\label{prop:adj}
Let $\tr \colon H^2(A)\to H^2(A)$ be the transfer
operator associated to an analytic expanding circle map $\tau$, 
with $A\in \mathcal{A}$ as in \eqref{eq:AA'}. 
%Then the adjoint $\tr^*$ of $\tr$ is conjugated to 
Then the isomorphism $J$ conjugates the adjoint $\tr^*$ of $\tr$ to
%the representation $\tr'$ with respect to the natural decomposition,
%$\mathcal{L}'\colon H^2(D_r)\oplus H^2(D^\infty_R)\to H^2(D_r)\oplus H^2(D^\infty_R)$ 
\[\tr'\colon \monster \to \monster\]
given by 
\begin{equation}\label{eq:blockmatrix}
\mathcal{L}'=\left(
\begin{array}{cc}
\Pi^{(r)}_+C^{(r)}
 & \Pi^{(R)}_{+} C^{(R)} \\
\Pi^{(r)}_{-} C^{(r)} & \Pi^{(R)}_{-}C^{(R)} \\
\end{array}
\right) ,
\end{equation}
%is isomorphically conjugated to the adjoint $\tr^*:H^2(A)^*\to H^2(A)^*$ 
that is $\tr' = J^{-1}\tr^*J$.
\end{prop}

\begin{proof} 
We want to show that 
$\tr^*J = J \tr'$, that is, 
\begin{equation}\label{eq:LJJL}
(\tr^*J(h_1,h_2))(f) = (J\tr'(h_1,h_2))(f)
\end{equation}
for all $(h_1,h_2)\in \monster$  and $f\in H^2(A)$. For any such $(h_1,h_2)$ and $f$, 
the adjoint property yields
\begin{eqnarray*}
(\tr^*J(h_1,h_2))(f) &=& (J(h_1,h_2))(\tr f) \\
&=& \frac{1}{2\pi i}\int_{\mathbb{T}_r} (\mathcal{L}f)(z) h_1(z) \,dz + 
 \frac{1}{2\pi i} \int_{\mathbb{T}_R} (\mathcal{L}f)(z)h_2(z) \,dz \ . 
\end{eqnarray*}

Now let a basis for $\monster$ be given by 
$\mathcal{P} = \{(p_n, 0): n\in \mathbb{N}_0\}\cup\{(0,p_{-n}): n\in \mathbb{N}\}$
with $p_n(z) = z^n$, where $p_n\in H^2(D_r)$ if $n\geq 0$ and $p_n\in H^2_0(D^\infty_R)$ if
$n<0$. 

Take $f \in \mathcal{E}$, where $\mathcal{E}$ % = \{e_m:m\in \mathbb{Z}\}$ 
is the basis for $H^2(A)$ given by \eqref{eq:en}.
For $n\in \mathbb{N}_0$ and $(h_1,h_2) = (p_n,0)\in \mathcal{P}$ we get
%and $f = e_m\in \mathcal{E}$.  
\begin{eqnarray*}
%\allowdisplaybreaks
(\tr^*J(h_1,0))(f) &=& \frac{1}{2\pi i}\int_{\mathbb{T}_r} (\mathcal{L}f)(z) h_1(z) \,dz \\
&\stackrel{(a)}{=}& \frac{1}{2\pi i}\int_{\mathbb{T}} (\mathcal{L}f)(z) h_1(z) \,dz\\
 &=& \frac{1}{2\pi i} \sum_{k=1}^{K} \int_{\mathbb{T}} \phi'_k(z)(f\circ \phi_k)(z)  h_1(z) \,dz \\
 &\stackrel{(b)}{=}& \frac{1}{2\pi i} \sum_{k=1}^{K} \int_{\phi_k(\mathbb{T})} f(w) (h_1 \circ \tau)(w) \,dw \\
&\stackrel{(c)}{=}& \frac{1}{2\pi i} \int_{\mathbb{T}} f(w) (h_1 \circ \tau)(w) \,dw\\
&\stackrel{(d)}{=}& \frac{1}{2\pi i} \int_{\mathbb{T}_r} f(w) (h_1 \circ \tau)(w) \,dw\,.
\end{eqnarray*}
Here, equalities $(a)$ and $(d)$ follow since the integrands are analytic on
$A$, equality 
% $(a)$ and $(d)$ follow from Lemma \ref{lem:homotopy},
$(b)$ follows by change of variables with $w = \phi_k(z)$ and
$\tau(w)=z$, and equality
$(c)$ is a consequence of 
the fact that $\bigcup_k \phi_k(\mathbb{T}) = \mathbb{T}$ up to measure
zero.
Then, by the definition of $\Pi^{(r)}_+$ and $\Pi^{(r)}_-$,
\begin{align*}
(\tr^*J&(h_1,0))(f) \\
%2\pi i \cdot  l(\mathcal{L}f) &=&   
=&  
\frac{1}{2\pi i}\int_{\mathbb{T}_r} f(w) (\Pi^{(r)}_+(h_1\circ \tau))(w) \,dw +
\frac{1}{2\pi i}\int_{\mathbb{T}_r} f(w) (\Pi^{(r)}_-(h_1\circ \tau))(w) \,dw \\
=& 
\frac{1}{2\pi i}\int_{\mathbb{T}_r} f(w) (\Pi^{(r)}_+(h_1\circ \tau))(w) \,dw +
\frac{1}{2\pi i}\int_{\mathbb{T}_R} f(w) (\Pi^{(r)}_-(h_1\circ \tau))(w) \,dw \\
=& 
(J\tr'(h_1,0))(f). 
\end{align*}
The penultimate equality follows 
%by contour deformation %from Lemma \ref{lem:homotopy} and 
from the fact that 
$\Pi^{(r)}_-(h_1\circ \tau) \in H_0^2(D^\infty_r)$.

Analogously, for $n\in \mathbb{N}$ and $(h_1,h_2) = (0,p_{-n})\in \mathcal{P}$, the same
argument 
%as above with $r$ replaced by $R$ 
shows
\[(\tr^*J(0,h_2))(f) = (J\tr'(0,h_2))(f).\]
Hence, for $f\in \mathcal{E}$, by linearity \eqref{eq:LJJL} holds for all finite linear 
combinations of basis elements $(h_1,h_2)$ in $\mathcal{P}$. Since these form a dense subspace of $\monster$, and
$\tr^*$, $\tr'$ and $J$ are continuous operators, equality \eqref{eq:LJJL} holds for 
all $(h_1,h_2)\in \monster$ and  $f\in \mathcal{E}$.
By continuity, this extends to all 
%Similar argument now applied to all finite linear combinations of 
%elements in $\mathcal{E}$ implies that \eqref{eq:LJJL} holds
$f\in H^2(A)$, which completes the proof.
\end{proof}

\begin{rem}
 Lemma \ref{lem:Cr_compact} and continuity of the projection 
 operators in \eqref{eq:P+P-} imply that $\tr'$ is compact. Note, however, 
 that this also follows from compactness of $\tr$ guaranteed by the choice
 of $A$ in \eqref{eq:AA'}.
\end{rem}

%%%%%%%%%%%%%%%%%%%%%%%%%%%%%%%%%%%%%%%%%%%%%%
%%%%%%%%%%% BLASCHKE PRODUCTS %%%%%%%%%%%%%%%%
%%%%%%%%%%%%%%%%%%%%%%%%%%%%%%%%%%%%%%%%%%%%%%
\section{Spectrum for Blaschke products}\label{sec:Blaschke}
% The last sections were devoted to transfer operators $\mathcal{L}$ 
% associated with
% analytic expanding circle maps and to deriving a convenient 
% representation of the 
% corresponding adjoint operators. 
% The aim of this section is to use this 
% representation 
% (Proposition \ref{prop:adj}) to obtain the full spectrum of 
% $\mathcal{L}$ for 
% finite Blaschke products,
% a class of circle maps defined as follows. 

Having discussed transfer operators $\mathcal{L}$ 
associated with
analytic expanding circle maps and a convenient 
representation of the 
corresponding adjoint operators (Proposition \ref{prop:adj}), 
we shall now use this 
representation 
to obtain the full spectrum of 
$\mathcal{L}$ for 
finite Blaschke products,
a class of circle maps defined as follows. 
\begin{defn}
For $n\geq2 $, let $\{a_1,\ldots,a_n\}$ be a 
finite set of complex numbers in the open unit disk $\mathbb{D}$.
 A \textit{finite Blaschke product} is a map of the form
$$ B(z)= C\prod_{i=1}^{n} \frac{z-a_i}{1-\overline{a_i}z}, $$
where $|C|=1.$
\end{defn}
It follows from the definition that
\begin{enumerate}[(i)]
 \item $B$ is a meromorphic function on $\hat{\mathbb{C}}$ with zeros $a_i$ and 
 poles $1/\overline{a}_i$;
  \item $B$ is holomorphic on a neighbourhood of $\overline{\mathbb{D}}$ 
  with  $B(\mathbb{D})=\mathbb{D}$ and $B(\mathbb{T})=\mathbb{T}$.
\end{enumerate}
Note also that a function  
$f$ is holomorphic on an open neighbourhood of $\overline{\mathbb{D}}$ 
with  $f(\mathbb{T})=\mathbb{T}$ if and only if $f$ is a finite
Blaschke product (see, for example, \cite[Exercise 6.12]{Burckel}). 

Let $\tau\colon\mathbb{T}\to\mathbb{T}$ denote the restriction of a 
finite Blaschke
product $B$ to $\mathbb{T}$.  A short calculation shows that 
$\tau$ is expanding if $\sum_{i=1}^{n}(1-|a_i|)/(1+|a_i|)>1$ (see 
\cite[Corollary to Prop.~1]{Martin1983} for details). 
Expansiveness of $\tau$ can also be expressed in terms of the nature
of the fixed points of $B$, as the following result shows. 
\begin{prop}\label{prop:BlaschkeProp}
Let $B$ and $\tau$ be as above. Then the following conditions are equivalent.
\begin{enumerate}[(a)]
  \item $|\tau'(z)|>1$ for all $z\in \mathbb{T}$. 
  \item 
  $B$ has exactly $n-1$ fixed points on $\mathbb{T}$, which are repelling, and
        two fixed points $z_0\in \mathbb{D}$ and  
        $\hat{z}_0=1/\overline{z}_0 \in 
        \hat{\mathbb{C}}\setminus\overline{\mathbb{D}}$,  
        which are attracting.   
 \end{enumerate}
\end{prop}
\begin{proof}
See \cite[Prop. 2.1]{PUJALS} and \cite{Tischler1999}.
\end{proof}

Crucial for the proof of our main theorem is the notion of  
a composition operator, which we briefly recall.   
\begin{defn}
Let $U$ be an open region in $\hat{\mathbb{C}}$. 
If $\psi\colon U\to U$ is holomorphic, then 
$C_{\psi}\colon \operatorname{Hol}(U)\to \operatorname{Hol}(U)$ defined by $C_{\psi}f = f\circ\psi$ 
is called a \textit{composition operator (with symbol $\psi$)}. 
\end{defn}
Note that in the 
literature the term `composition operator' is mostly used in the context of 
holomorphic functions. The operators in~\eqref{eq:Crho} 
considered on $L^2(\mathbb{T}_r)$ do not formally fall into this
category, 
but will turn out to be composition
operators for Blaschke product symbols. 

We are now able to state our main result.

\begin{theorem}\label{thm:main}
%\emph{(Main Theorem)}\label{thm:main}\\
Let $B$ be a finite Blaschke product such that 
$\tau = B\vert_{\mathbb{T}}$ %:\mathbb{T}\to\mathbb{T}$ 
is an analytic expanding circle map.
Then
\begin{enumerate}[(a)]
 \item the transfer operator $\mathcal{L}\colon H^2(A) \to H^2(A)$ associated with 
$\tau$ is well defined and compact for some annulus $A\in \mathcal{A}$, and
\item 
the spectrum of $\mathcal{L}\colon H^2(A) \to H^2(A)$ is given by
        \begin{equation}\label{eq:specL}
\sigma(\mathcal{L}) =  
\left \{ \lambda(z_0)^n: n\in\mathbb{N}_0 \right \} \cup 
\left \{\lambda (\hat{z}_0) ^n : n\in\mathbb{N}\right \} 
\cup \left \{0 \right \},        
        \end{equation}
where $\lambda(z_0)$ and $\lambda(\hat{z_0}) = \overline{\lambda(z_0)}$  are the multipliers\footnote{
Recall that the multiplier $\lambda(z^*)$ of a fixed point $z^*$ of 
a rational map 
$R$ is given by $R'(z^*)$ if $z^*\in \mathbb{C}$ and $1/R'(z^*)$ 
if $z^* = \infty$. 
For Blaschke products the 
equality $\lambda(\hat{z_0}) = \overline{\lambda(z_0)}$
follows from a straightforward calculation.} 
of the
unique fixed points $z_0$ and $\hat{z_0}$ of $B$ in 
$\mathbb{D}$ and 
$\hat{\mathbb{C}}\setminus\overline{\mathbb{D}}$, respectively.
\end{enumerate}
\end{theorem}

\begin{proof}
The first assertion is obvious, as $\tau$ is an analytic expanding circle map
and we can choose $A = A_{r,R}\in \mathcal{A}$ as in \eqref{eq:AA'} such that $\tr$ 
is well defined and compact by the results in Section \ref{sec:CircTrans}.

For the second claim, 
we will use the fact that the spectrum of $\tr$ coincides with that of its 
adjoint $\tr^*$, which 
together with the structure of the representation $\tr'$ of $\tr^*$ 
will allow us to deduce \eqref{eq:specL}. 

We start by observing that for the chosen $A$ we have 
$B(\partial A ) \cap \cl{A} = \emptyset$, 
as well as $B(D_r)\cc D_r$ and $B(D^\infty_R)\cc D^\infty_R$.
It follows that $f\circ B \in H^2(D_r)$ for any $f\in H^2(D_r)$, 
and $f\circ B \in H^2(D^\infty_R)$ for any $f\in H^2(D^\infty_R)$, 
so that $C^{(r)}_Bf = f\circ B$ and $C^{(R)}_Bf=f\circ B$ %in \eqref{eq:Crho}  
define composition operators
on $H^2(D_r)$ and $H^2(D^\infty_R)$, respectively. 
It is a standard fact that $B(D_r)\cc D_r$ guarantees compactness
of $C^{(r)}_B$ (see, for example, \cite[pp.~128-129]{CowM1995}).  
Similarly for $C^{(R)}_B$.
It is also well known 
(see \cite[Lem.~7.10]{Mayer1991} or \cite[Thm.~7.20]{CowM1995})
that all eigenvalues of a 
compact composition operator $C_\psi$ are simple and are given by the 
non-negative integer powers of the multiplier of the unique attracting 
fixed point of $\psi$.
Hence, 
\[\sigma(C^{(r)}_B) = 
\left \{ \lambda(z_0)^n : n\in\mathbb{N}_0\right \}\cup\{0\}
\]
and 
\[\sigma(C^{(R)}_B) = 
\left \{ \lambda(\hat{z_0})^n :n\in\mathbb{N}_0 \right \}
\cup \{0 \},\]
where $z_0$ and $\hat{z_0}$ are
the unique attracting fixed points of 
$B$ in $D_r$ and $D^\infty_R$, respectively
(see Proposition \ref{prop:BlaschkeProp}).

We now explain how to use these observations to determine the
spectrum of $\tr'$ given in \eqref{eq:blockmatrix}.
%and consequently to $\tr^*$ and $\tr$.
Note that $\Pi^{(r)}_+ C^{(r)}_B = C^{(r)}_B$ 
and $\Pi^{(r)}_-C^{(r)}_B = 0$, where $\Pi^{(r)}_+$ and $\Pi^{(r)}_-$
are the projection operators in \eqref{eq:P+P-}.
%defined in the paragraph preceding Proposition \ref{prop:adj}. 
Thus the operator $\tr'$ is given by
\begin{equation}\label{eq:blockmatrix1}
\tr'=\left(
\begin{array}{cc}
C^{(r)}_B
 & \Pi^{(R)}_+ C^{(R)}_B \\
 0 & \Pi^{(R)}_- C^{(R)}_B \\
\end{array}
\right) .
\end{equation}
In particular, $\tr'$ leaves $H^2(D_r)\oplus \{0\}$ invariant.
The operator $\Pi^{(R)}_- C^{(R)}_B$ is not a composition operator
on $H^2_0(D^\infty_R)$, but we can relate its spectrum to the spectrum of
$C^{(R)}_B$ on $H^2(D^\infty_R)$.
More precisely,
\begin{equation}\label{eq:specC_PC}
\sigma(\PRminus C^{(R)}_B) = \sigma(C^{(R)}_B)\setminus \{1\},
\end{equation}
as we shall see below.
Then, using \eqref{eq:specC_PC} the assertion of the theorem follows,
since 
\begin{eqnarray*}
\sigma(\tr') &=& \sigma(C^{(r)}_B) \cup \sigma(\PRminus C^{(R)}_B) \\
 &=& 
\left \{\lambda(z_0)^n :n\in\mathbb{N}_0\right \} \cup 
\left \{\lambda(\hat{z_0}) ^n :n\in\mathbb{N} \right \} 
\cup \left \{0 \right \},
\end{eqnarray*}
and $\sigma(\tr) = \sigma(\tr^*) = \sigma(\tr')$.

It remains to prove \eqref{eq:specC_PC}. For brevity,
we drop the superscript $(R)$ from $\PRminus$, $\PRplus$ and $C^{(R)}_B$
since we only consider functions in $H^2(D^\infty_R)$ in what follows.
Observe that for $f\in H^2(D^\infty_R)$, we have $(\Pi_+f)(z) = f(\infty)$,
which implies
\begin{align}\label{eq:PCP}
 C_B\Pi_+ = \Pi_+ \quad \text{and} \quad 
 \Pi_-C_B= \Pi_-C_B\Pi_-\,.
\end{align}
Note that $1$ is an eigenvalue of $C_B$ if and only if the corresponding eigenfunction
is constant.
Take $\mu\in \sigma(C_B)$ with $\mu(1-\mu)\neq 1$. 
Since $C_B$ is compact, there is a non-zero 
$f\in H^2(D^\infty_R)$ with 
$C_Bf=\mu f$. The second equality in \eqref{eq:PCP} now implies
$\Pi_-C_B\Pi_- f = \mu \Pi_-f$. But since $\mu \neq 1$ the
eigenvector $f$ is non-constant,
so we have $0\neq\Pi_-f\in H_0^2(D^\infty_R)$ and 
thus $\mu\in \sigma(\Pi_-C_B)$.

To show the converse inclusion, take $\mu \in \sigma(\Pi_-C_B)$ with
$\mu\neq 0$. Since $\Pi_-C_B$ is compact, there is 
a non-zero $f\in H_0^2(D^\infty_R)$
with 
$\Pi_-C_B f= \mu f$. 
First observe that\footnote{%If $1\in \sigma_p(\Pi_-C_B)$ 
%then 
In order to see this, note that otherwise
$\Pi_-C_B f = f$, which implies $f\circ B - f = \operatorname{const}$.
However $f(B(\hat{z_0}))-f(\hat{z_0})=0$, which implies 
$f = f\circ B$. Thus 
$f =\operatorname{const}$, so $\Pi_-C_Bf=0$, contradicting the fact
that $\mu\neq 0$.}
$\mu \neq 1$. Next we note that if $\mu(\mu-1)\neq 0$, then 
$(1-\mu)f-\Pi_+C_Bf\neq 0$ (for otherwise $f$ would be zero). 
Finally, we use 
\eqref{eq:PCP} to show that $(1-\mu)f-\Pi_+C_Bf$ is an eigenfunction
of $C_B$ with eigenvalue $\mu$:
\begin{equation*}
\begin{split}
 C_B \left((1-\mu)f-\Pi_+C_Bf \right) &= (1-\mu)(C_Bf +(\mu f -\Pi_-C_Bf))-C_B\Pi_+C_Bf\\
& =\mu(1-\mu)f+(1-\mu)(I-\Pi_-)C_Bf -\Pi_+C_Bf\\ 
&= \mu \left((1-\mu)f-\Pi_+C_Bf \right).
\end{split}
\end{equation*}
Thus $\sigma(\Pi_- C_B) = \sigma(C_B)\setminus \{1\}$ as claimed.\qedhere
\end{proof}
The following examples illustrate our main result. 
\begin{example}\label{ex:DoublMap}
%\begin{enumerate}[(i)]
 The map $B(z)=z^n$ for $n\geq 2$ has two attracting fixed points 
 $z_0 =0$ and $\hat{z_0}=\infty$ with $\lambda(z_0)=\lambda(\hat{z}_0)=0$. 
        %$B'(z_0)=1/B'(\hat{z_0})=0.$
        Thus $\sigma(\tr) = \{0,1\}$.
\end{example}
Curiously enough, these are not the only examples for which 
$\sigma(\tr)=\{0,1\}$, as the next example shows.  
\begin{example}
Let $B(z) = z^2(z-b)/(1-bz)$ for $b\in (-1,1)$. Then
$B|_{\mathbb{T}}$ is an expanding $3$-to-$1$ circle map. %has 
As the multipliers of the attracting fixed points of $B$ are
vanishing, just as 
in Example \ref{ex:DoublMap}, we get $\sigma(\tr) = \{0,1\}$. 
% Moreover, as can be shown with \cite[Lemma 4.2]{Nr.2}, 
% the spectrum of the transfer operator $\tr_I$ associated to a particular interval map
% arising from $B|_{\mathbb{T}}$ is
% \begin{equation*}
%  \sigma(\tr_I) = \{1,0\}\cup \left\{\left(\frac{1+b}{3+b}\right)^n: 
%  n\in \mathbb{N}\right\}.
% \end{equation*}
% Obviously,
% $(1+b)/(3+b)$ tends to $0$ as $b \to -1$. 
% This shows that there are expanding interval maps with a finite number of branches
% for which the subleading eigenvalue
% of $\tr_I$ can be made arbitrarily small.
\end{example}

\begin{example} For the family of maps  $B(z)=z(\mu - z)/(1-\overline{\mu} z)$
considered in \cite{Nr.2}, the restriction
$B|_{\mathbb{T}}$ is an expanding circle map for any $\mu \in \mathbb{D}$.
The attracting fixed points are $z_0 =0$ and $\hat{z_0}=\infty$ with 
$\lambda(z_0) = \mu$ and $\lambda(\hat{z}_0) = \overline{\mu}$.
%$B'(z_0)= \lambda$ and $1/B'(\hat{z_0})=\overline{\lambda}$.
Thus  
\[ \sigma(\tr) =  
\left \{\mu^n : n\in\mathbb{N}_0\right \} \cup 
 \{\overline{\mu}^n :n\in\mathbb{N}\} 
\cup \left \{0 \right \}\,.
\]
\end{example}

%%%%%%%%%%%%%% Bibliography %%%%%%%%%%%%%%

%\bibliographystyle{abbrv}
%\bibliography{Nr3_3Draft.bib}

\end{document}